\documentclass[12pt]{article}%
\usepackage{graphicx}
\usepackage{color}
\usepackage{amsmath}
\usepackage{a4}
\usepackage{amsfonts}
\usepackage{amssymb}
\usepackage[all]{xy}
\providecommand{\U}[1]{\protect \rule{.1in}{.1in}}
\newtheorem{theorem}{Theorem}

\newtheorem{corollary}[theorem]{Corollary}

\newtheorem{definition}[theorem]{Definition}

\newtheorem{lemma}[theorem]{Lemma}
\newtheorem{notation}[theorem]{Notation}

\newtheorem{proposition}[theorem]{Proposition}
\newtheorem{remark}[theorem]{Remark}

\newenvironment{proof}[1][Proof]{\textbf{#1.} }{\  \rule{0.5em}{0.5em}}

\begin{document}

\title{The geometry of Euclidean surfaces with conical singularities}
\author{Charalampos Charitos$\dagger$, Ioannis Papadoperakis$\dagger$
\and and Georgios Tsapogas$\ddagger$\\$\dagger$Agricultural University of Athens \\and $\ddagger$University of the Aegean}
\maketitle

\begin{abstract}
The geometry of closed surfaces equipped with a Euclidean metric with finitely
many conical points of arbitrary angle is studied. The main result is that the
image of a non-closed geodesic has $0$ distance from the set of conical
points. Dynamical properties for the space of geodesics are also proved.

\end{abstract}

\section{Preliminaries}

Let $S$ be a closed surface of genus $\geq2$ equipped with a euclidean metric
with finitely many \textit{conical singularities} (or \textit{conical
points}), say $s_{1},...,s_{n}.$ Every point which is not conical will be
called a \textit{regular point} of $S.$ Denote by $\theta(s_{i})$ the angle at
each $s_{i}.$

In this work we show that if $g$ is a non-closed geodesic in $S,$ then its
image has $0$ distance from the set of conical points $\left \{  s_{1}%
,...,s_{n}\right \}  $ (see Theorems \ref{no_strips} and \ref{dense_small}
below). This is done by analyzing the existence of flat strips in the
universal cover $\widetilde{S}$ of $S.$ In fact we show that a geodesic line
$\widetilde{g}$ in $\widetilde{S},$ whose projection is not a closed geodesic
in $S,$ cannot be contained in a (Euclidean) flat strip. In particular, any
two elements $\xi,\eta$ in $\partial \widetilde{S}$ determine a unique geodesic
line in $\widetilde{S}$ provided that $\xi,\eta$ are not the limit points of a
hyperbolic element $\phi \in \pi_{1}(S).$ We then use Theorem \ref{no_strips} to
show a classical result in this setup, namely, that closed geodesics form a
dense subset in the space $GS$ of geodesics in $S.$ In the final section we
generalize to surfaces with $\theta(s_{i})\in(0,+\infty)\setminus \left \{
2\pi \right \}  $ by employing saddle connections.

We write $C\left(  v,\theta \right)  $ for the standard cone with vertex $v$
and angle $\theta,$ namely, $C\left(  v,\theta \right)  $ is the set $\left \{
\left(  r,t\right)  :0\leq r,t\in \mathbb{R}/\theta \mathbb{Z}\right \}  $
equipped with the metric $ds^{2}=dr^{2}+r^{2}dt^{2}.$

\begin{definition}
A Euclidean surface with conical singularities $s_{1},...,s_{n}$ is a closed
surface $S$ equipped with a metric $d\left(  \cdot,\cdot \right)  $ such that at

\begin{itemize}
\item Every point $p\in S\setminus \left \{  s_{1},...,s_{n}\right \}  $ has a
neighborhood isometric to a disk in the Euclidean plane

\item Each $s_{i}\in \left \{  s_{1},...,s_{n}\right \}  $ has a neighborhood
isometric to a neighborhood of vertex $v$ of the standard cone $C\left(
v,\theta \left(  s_{i}\right)  \right)  .$
\end{itemize}
\end{definition}

From now on and until the final section we assume that the angle $\theta
(s_{i})$ of the conical point $s_{i}$ satisfies $\theta(s_{i})\in(2\pi
,+\infty).$

Clearly, the metric on $S$ is a length metric and the surface $S\ $will be
written $e.s.c.s.$ for brevity. Note that if the genus $g$ is $\geq2,$ such
Euclidean structures exist, see \cite{[Tro]}.

\begin{definition}
A geodesic segment is an isometric map $h:[a,b]\rightarrow S.$ If $x=h(a)$ and
$y=h(b)$ then a geodesic segment joining $x$ and $y$ will be denoted by
$[x,y].$ \newline Let $I=[0,+\infty)$ or $I=(-\infty,+\infty).$ A geodesic
line (resp. geodesic ray) in $S$ is a local isometric map $h:I\rightarrow S$
where $I=\left(  -\infty,+\infty \right)  $ (resp. $I=\left[  0,+\infty \right)
$ ).\newline A closed geodesic is a local isometric map $h:I\rightarrow S$
which is a periodic map.\newline A metric space is called geodesic if every
two points can be joined by a geodesic segment.
\end{definition}

\noindent Observe that a geodesic segment is allowed to contain in its image a
conical singularity.

As every locally compact, complete length space is geodesic (see Th. 1.10 in
\cite{[GLP]}) we immediately have

\begin{proposition}
If $S\ $is a $e.s.c.s.$ then $S$ is a geodesic space.
\end{proposition}

\section{The universal covering and the limit set}

Assume that $S$ is a closed $e.s.c.s.$ of genus $g\geq2$ and denote by $d$ the
euclidean metric on $S.$ Let $\widetilde{S}$ be the universal covering of $S$
and let $p_{S}:\widetilde{S}\rightarrow S$ be the covering projection.
Obviously, the universal covering $\widetilde{S}$ is homeomorphic to
$\mathbb{R}^{2}$ and by requiring $p_{S}$ to be a local isometric map we may
lift $d$ to a metric $\widetilde{d}$ on $\widetilde{S}$ so that $(\widetilde
{S},$ $\widetilde{d})$ becomes an $e.s.c.s.$ There is a discrete group of
isometries $\Gamma$ of $\widetilde{S}$ which is isomorphic to $\pi_{1}(S),$
acting freely on $\widetilde{S}$ so that $S=\widetilde{S}/\Gamma.$

The group $\Gamma$ with the word metric is \textit{hyperbolic in the sense of
Gromov} since $\Gamma$ is isomorphic to $\pi_{1}(S).$ On the other hand,
$\Gamma$ acts co-compactly on $\widetilde{S};$ this implies that
$\widetilde{S}$ is itself a hyperbolic space in the sense of Gromov (see for
example Th. 4.1 in \cite{[CDP]}) which is complete and locally compact. Hence,
$\widetilde{S}$ is a proper space i.e. each closed ball in $\widetilde{S}$ is
compact (see \cite{[GLP]} Th. 1.10). Therefore, the visual boundary
$\partial_{vis}\widetilde{S}$ of $\widetilde{S}$ is defined by means of
geodesic rays and is homeomorphic to $\mathbb{S}^{1}$ (see \cite{[CDP]},
p.19). Furthermore, as all conical singularities in $S$ are assumed to have
angle $>2\pi,$ we deduce that $S$ has curvature $\leq0,$ that is, $S$
satisfies CAT(0) inequality locally (see for example \cite[Theorem
3.15]{[Pau]}). Since $\widetilde{S}$ is simply connected it follows that
$\widetilde{S}$ satisfies CAT(0) globally, i.e., $\widetilde{S}$ is a Hadamard
space (for the definition and properties of CAT(0) spaces see \cite{[Bal]}).
Note that since $\widetilde{S}$ is a CAT(0) space, geodesic lines and geodesic
rays in $\widetilde{S}$ are global isometric maps.

In the next proposition we state the following important property of
$\partial_{vis}\widetilde{S}.$ For a proof see Proposition 2.1 in \cite{[CDP]}.

\begin{proposition}
\label{existence rays, lnes}For every pair of points $x\in$ $\widetilde{S},$
$\xi \in \partial_{vis}\widetilde{S}$ (resp. $\eta,\xi \in \partial_{vis}%
\widetilde{S})$ there is a geodesic ray $r:[0,\infty)\rightarrow \widetilde
{S}\cup \partial_{vis}\widetilde{S}$ (resp. a geodesic line $l:(-\infty
,\infty)\rightarrow \widetilde{S}\cup \partial_{vis}\widetilde{S})$ such that,
$r(0)=x,$ $r(\infty)=\xi$ (resp. $l(-\infty)=\eta,$ $l(\infty)=\xi).$
\end{proposition}

Since $\widetilde{S}$ is a hyperbolic space in the sense of Gromov, the
isometries of $\widetilde{S}$ are classified as elliptic, parabolic and
hyperbolic \cite{[Gr]}. On the other hand, $\Gamma$ is a hyperbolic group,
thus $\Gamma$ does not contain parabolic elements with respect to its action
on its Cayley graph (see Th. 3.4 in \cite{[CDP]}). From this, it follows that
all elements of $\Gamma$ are hyperbolic isometries of $\widetilde{S}.$
Therefore, for each $\varphi \in \Gamma$ and each $x\in \widetilde{S}$ the
sequence $\varphi^{n}(x)$ (resp. $\varphi^{-n}(x))$ has a limit point
$\varphi(+\infty)$ (resp. $\varphi(-\infty))$ when $n\rightarrow+\infty$ and
$\varphi(+\infty)\neq \varphi(-\infty).$ The point $\varphi(+\infty)$ is called
\textit{attractive} and the point $\varphi(-\infty)$ \textit{repulsive} point
of $\varphi.$

The \textit{limit set} $\Lambda(\Gamma)$ of $\Gamma$ is defined to be
$\Lambda(\Gamma)=\overline{\Gamma x}\cap \partial_{vis}\widetilde{S},$ where
$x$ is an arbitrary point in $\widetilde{S}.$ Since the action of $\Gamma$ on
$\widetilde{S}$ is co-compact, it is a well known fact that $\Lambda
(\Gamma)=\partial_{vis}\widetilde{S},$ and hence $\Lambda(\Gamma
)=\mathbb{S}^{1}.$ Note that the action of $\Gamma$ on $\widetilde{S}$ can be
extended to $\partial_{vis}\widetilde{S}$ and that the action of $\Gamma$ on
$\partial_{vis}\widetilde{S}$ $\times \partial_{vis}\widetilde{S}$ is given by
the product action.

Denote by $F_{h}$ the set of points in $\partial_{vis}\widetilde{S}$ which are
fixed by hyperbolic elements of $\Gamma.$ Since $\Lambda(\Gamma)=\partial
_{vis}\widetilde{S},$ the following three results can be derived from
\cite{[Coo]}.

\begin{proposition}
\label{dense1} The set $F_{h}$ is $\Gamma-$invariant and dense in
$\partial_{vis}\widetilde{S}.$
\end{proposition}

\begin{proposition}
\label{dense2} There exists an orbit of $\Gamma$ dense in $\partial
_{vis}\widetilde{S}\times \partial_{vis}\widetilde{S}.$
\end{proposition}

\begin{proposition}
\label{dense3}The set $\{ \left(  \phi(+\infty),\phi(-\infty)\right)  :\phi
\in \Gamma$ is hyperbolic\} is dense in $\partial_{vis}\widetilde{S}%
\times \partial_{vis}\widetilde{S}.$
\end{proposition}

\begin{lemma}
\label{existence closed geodesic}Let $\varphi$ be an element of $\Gamma$ and
let $\eta=\varphi(-\infty)$ and $\xi=\varphi(\infty)$ be the repulsive and
attractive points of $\varphi$ in $\partial \widetilde{S}.$ Then any geodesic
line $c$ joining $\eta$ and $\xi$ projects to a closed geodesic in $S.$
\end{lemma}

\begin{proof}
$S$ is a $CAT(0)$ space and $\varphi$ is an axial isometry (following
Definition 3.1 of \cite{[Bal]}). Therefore, from Proposition 3.3, p. 31 of
\cite{[Bal]}, there is an axis $c_{0}$ of $\varphi$ in $\widetilde{S}$ which
projects to a closed geodesic in $S.$

Let $c$ be a geodesic line of $\widetilde{S}$ joining the points $\eta,$
$\xi.$ Then $c$ and $c_{0}$ are parallel in $\widetilde{S}$ i.e. they bound a
flat strip in $\widetilde{S},$ (see Pr. 5.8, p. 25 in \cite{[Bal]}).
Therefore, by Proposition 3.3 of \cite{[Bal]}, $c$ is also an axis of
$\varphi$ and thus it projects to a closed geodesic in $S.$
\end{proof}

\begin{figure}[ptb]
\begin{center}
\includegraphics[scale=2.3]
{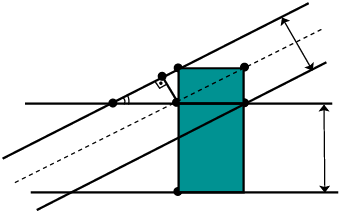}
\end{center}
\par
\begin{picture}(22,12)
\put(190,186){$E$}
\put(207,194){$A$}
\put(280,194){$D$}
\put(162,151){$\theta$}
\put(140,158){$F$}
\put(205,135){$B$}
\put(289,135){$C$}
\put(205,36){$G$}
\put(378,90){$\delta$}
\put(355,200){$\varepsilon$}
\end{picture}
\caption{The quadrangle in the intersection of two flat strips}%
\label{quardangle}%
\end{figure}
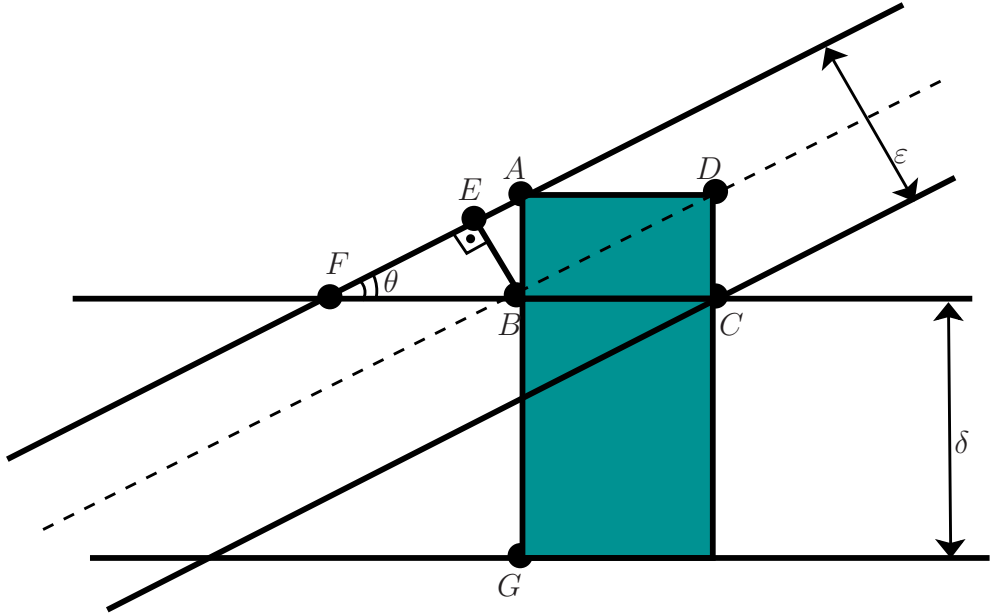

\section{Flat strips and closed geodesics}

Let $GS$ be the space of all local isometric maps $\gamma:\mathbb{R}%
\rightarrow S.$ The image of such a $\gamma$ will be referred to as a geodesic
in $S.$ We recall here that geodesics are allowed to contain conical
singularities. The geodesic flow is defined by the map

\begin{center}
$\Phi:\mathbb{R}\times GS\rightarrow GS$
\end{center}

\noindent where the action of $\mathbb{R}$ is given by right translation, i.e.
for each $t\in \mathbb{R}$ and $\gamma \in GS,$ $\Phi(t,\gamma)=t\cdot \gamma,$
where $t\cdot \gamma:\mathbb{R}\rightarrow S$ is the geodesic defined by
$t\cdot \gamma(s)=\gamma(t+s),$ $s\in \mathbb{R}.$

Consider also the space $G\widetilde{S}$ of all isometric maps $g:\mathbb{R}%
\rightarrow \widetilde{S}.$ Both spaces $GS$ and $G\widetilde{S}$ are equipped
with the compact-open topology. Moreover the space $G\widetilde{S}$ with the
compact-open topology is metrizable and its metric is given by the formula
(see 8.3.B in \cite{[Gr]})%

\[
\left \vert g_{1}-g_{2}\right \vert _{G\widetilde{S}}=\int_{-\infty}^{+\infty
}\left \vert g_{1}(t)-g_{2}(t)\right \vert _{\widetilde{S}}e^{-|t|}dt.
\]

In the proof of Proposition \ref{u_ubar} below, we will need the fact that
$GS$ with the compact-open topology is sequentially compact. A proof is
included here for the reader's convenience.

\begin{proposition}
\label{epitelous} The space of geodesics $GS$ equipped with the compact-open
topology is sequentially compact.
\end{proposition}

\begin{proof}
Let $\left \{  g_{n}\right \}  _{n\in \mathbb{N}}$ be a sequence in $GS.$ As $S$
is compact we may assume, by passing to a subsequence if necessary, that
$\left \{  g_{n}\left(  0\right)  \right \}  _{n\in \mathbb{N}}$ converges to a
point $x_{0}\in S.$ Pick any $\widetilde{x}_{0}$ with $p\left(  \widetilde
{x}_{0}\right)  =x_{0}$ and find $r>0$ so that the open disk $D\left(
\widetilde{x}_{0},r\right)  $ projects via the covering map $p$ isometrically
onto $D\left(  x_{0},r\right)  .$ For each $n\in \mathbb{N},$ the geodesic
$g_{n}$ lifts uniquely to a geodesic $\widetilde{g}_{n}\in G\widetilde{S}$
with $\widetilde{g}_{n}\left(  0\right)  \in D\left(  \widetilde{x}%
_{0},r\right)  .$ By passing to a subsequence, we may assume that
\begin{equation}
\widetilde{g}_{N}\left(  0\right)  \in D\left(  \widetilde{x}_{0},\frac{1}%
{n}\right)  \mathrm{\  \ for\  \ all\  \ }N\geq n. \label{enani}%
\end{equation}
Fix a sequence $\left \{  t_{n}\right \}  $ converging to $\infty.$ For each
$n\in \mathbb{N},$ denote by $\sigma_{n}$ the (unique) geodesic segment
$\left[  \widetilde{x}_{0},\widetilde{g}_{n}\left(  t_{n}\right)  \right]  $
in $\widetilde{S}.$ Clearly, $d\left(  \widetilde{x}_{0},\widetilde{g}%
_{n}\left(  t_{n}\right)  \right)  \rightarrow \infty.$ By passing to a subsequence if 
necessary,  the sequence of geodesic segments $\left \{  \sigma
_{n}\right \}  $ converges uniformly on compact sets to a geodesic ray
$r_{\widetilde{x}_{0}}^{+}$ with initial point $\widetilde{x}_{0}$ (see \cite[Chapter II, Prop. 2.5]{[Bal]}).

Denote by $\tau_{n}$ the geodesic sub-segment of $\operatorname{Im}%
\widetilde{g}_{n}$ with initial point $\widetilde{g}_{n}\left(  0\right)  $
and length equal to the length of $\sigma_{n}.$ By a triangle inequality
argument and using property (\ref{enani}) above, the endpoints of $\sigma
_{N},\tau_{N}$ are of distance $\leq \frac{2}{n}$ for all $N\geq n.$

Since $\widetilde{S}$ is a CAT(0) space, the distance function $t\rightarrow
d\left(  \sigma_{N}\left(  t\right)  ,\tau_{N}\left(  t\right)  \right)  $ is
convex in $t$ (see \cite[Chapter I, Prop. 5.2]{[Bal]}). It follows that
\[
d\left(  \sigma_{N}\left(  t\right)  ,\tau_{N}\left(  t\right)  \right)
\leq \frac{2}{n}\mathrm{\  \ for\  \ all\  \ }N\geq n\mathrm{\mathrm{\  \ and\ for}%
\ all\  \ }t\in \left[  0,t_{n}\right]  .
\]
Combining the latter with the fact that $\sigma_{n}\rightarrow r_{\widetilde
{x}_{0}}^{+}$ we conclude that $\widetilde{g}_{n}|_{\left[  0,\infty \right)
}$ converges uniformly on compact sets to the geodesic ray $r_{\widetilde
{x}_{0}}^{+}.$

Working similarly we obtain that $\widetilde{g}_{n}|_{\left(  -\infty
,0\right]  }$ converges uniformly on compact sets to a geodesic ray
$r_{\widetilde{x}_{0}}^{-}.$ It remains to show that $r_{\widetilde{x}_{0}%
}^{-}\cup r_{\widetilde{x}_{0}}^{+}$ is a geodesic line or, equivalently, it
is a local geodesic at $\widetilde{x}_{0}.$ For, if for some $t^{+}>0,$
$t^{-}<0$ the geodesic segment $\left[  r_{\widetilde{x}_{0}}^{-}\left(
t^{-}\right)  ,r_{\widetilde{x}_{0}}^{+}\left(  t^{+}\right)  \right]  $ does
not contain $\widetilde{x}_{0}$ then for sufficiently large $N,$ the points
$\widetilde{g}_{N}\left(  t^{-}\right)  $ and $\widetilde{g}_{N}\left(
t^{-}\right)  $ would be joined by a path of length smaller that
$t^{+}+\left \vert t^{-}\right \vert $ violating the fact that $g_{N}$ is a
geodesic. This concludes the proof of the proposition.
\end{proof}

\begin{definition}
A \textit{ flat strip} in $\widetilde{S}$ is a subset of $\widetilde{S}$
isometric, with respect to the induced metric, to a strip $[0,\varepsilon
]\times \mathbb{R}$ in $\mathbb{R}^{2}$, for an appropriate $\varepsilon>0$
such that for each $r\in \left[  0,\varepsilon \right]  ,$ $\{r\} \times
\mathbb{R}$ is the image of a geodesic line. If in addition, for
$r=0,\varepsilon$ the distance $d\left(  \operatorname{Im}\left(  \{r\}
\times \mathbb{R}\right)  ,\left \{  s_{1},...,s_{n}\right \}  \right)  $ of the
images of the geodesic lines $\{r\} \times \mathbb{R}$ from the set of conical
points is $0,$ the flat strip will be called maximal flat strip. Observe that
a maximal flat strip may contain conical points only in its boundary lines
corresponding to $\{0\} \times \mathbb{R}$ and $\{ \varepsilon \} \times
\mathbb{R}$.\newline By substituting $\mathbb{R}$ with $\left[  0,\infty
\right)  $ and geodesic lines by geodesic rays, the notion of flat half strip
is defined.
\end{definition}

\begin{notation}
The number $\varepsilon>0$ will be called the width of the flat strip (resp.
flat half strip for geodesic rays). A flat strip of width $\varepsilon$ will
be denoted by $FS\left(  \varepsilon \right)  $ (resp. $\left.  FHS\left(
\varepsilon \right)  \right)  .$\newline For $\widetilde{g}$ in $G\widetilde
{S}$ we will be writing $\widetilde{g}\in FS\left(  \varepsilon \right)  $ to
indicate that $\widetilde{g}$ is identified with a line $\{r\} \times
\mathbb{R}$ in $[0,\varepsilon]\times \mathbb{R\equiv}FS\left(  \varepsilon
\right)  .$ \newline For $g$ in $GS$ we will be writing $g\in FS\left(
\varepsilon \right)  $ to indicate that there exists a flat strip $FS\left(
\varepsilon \right)  \subset \widetilde{S}$ such that for some lift
$\widetilde{g}$ of $g$, $\widetilde{g}\in FS\left(  \varepsilon \right)  .$
Similarly, we write $\widetilde{g}\in FHS\left(  \varepsilon \right)  $ and
$g\in FHS\left(  \varepsilon \right)  $ for geodesic rays. \newline If
$\widetilde{g}\in G\widetilde{S}$ we write $\varepsilon \left(  \widetilde
{g}\right)  $ for the $\sup \left \{  \varepsilon|\widetilde{g}\in FS\left(
\varepsilon \right)  \right \}  .$ Set $\varepsilon \left(  \widetilde{g}\right)
=0$ if there is no such $\varepsilon$ positive. Given $g\in GS,$ pick a lift
$\widetilde{g}\in G\widetilde{S}$ of $g$ and set $\varepsilon \left(  g\right)
:=\varepsilon \left(  \widetilde{g}\right)  .$ Clearly, $\varepsilon \left(
g\right)  $ does not depend on the choice of $\widetilde{g}.$ \newline By
writing $FS\left(  g\right)  $ (resp. $\left.  FS\left(  \widetilde{g}\right)
\right)  $ we mean the maximal flat strip of width $\varepsilon=\varepsilon
\left(  g\right)  $ (resp. $\left.  \varepsilon=\varepsilon \left(
\widetilde{g}\right)  \right)  $ containing $g$ (resp. $\widetilde{g}$).
Similarly, $FHS\left(  g\right)  $ and $FHS\left(  \widetilde{g}\right)  $ for
geodesic rays.
\end{notation}

Given a flat strip in $\widetilde{S}$ we may parametrize each geodesic line
$\{r\} \times \mathbb{R}$ by a mapping $\gamma^{r}:(-\infty,\infty
)\rightarrow \widetilde{S}$ such that $\left \{  \gamma^{r} (s) |r\in
[0,\varepsilon] \right \}  $ is identified with $[0,\varepsilon] \times \{s\} .$
We will call such a parametrization a \textit{normal parametrization } of the
given flat strip.

We show below (see Theorem \ref{no_strips}) that the image of a non-closed
geodesic in $S$ does not have positive distance from the set of conical
points. This implies that uniqueness of geodesic lines in $\widetilde{S}$ does
hold for all geodesic which do not project to closed geodesics in $S$ (cf
Lemma \ref{existence closed geodesic}). In other words, (maximal) flat strips
in $\widetilde{S}$ correspond precisely to closed geodesics in $S.$ In view of
the latter, we will use occasionally the following quotient space of geodesics.

For each flat strip $E$ isometric to $\left[  0,\varepsilon \right]
\times \mathbb{R},$ we identify each family of geodesics $g^{r},r\in \left[
0,\varepsilon \right]  $ which forms a normal parametrization of $E$ to a
unique geodesic line $g_{E}$ and denote the resulting quotient space by
$G_{0}\widetilde{S}.$ Thus, for every $r\in \lbrack0,\varepsilon]$ the points
$g^{r}(s)$ are identified to $g_{E}(s),$ where $s\in \mathbb{R}.$ Moreover, if
a family $g^{r},r\in \left[  0,\varepsilon \right]  $ forms a normal
parametrization of $E,$ so does the family $t\cdot g^{r},r\in \left[
0,\varepsilon \right]  .$ Thus, the geodesic flow on $GS$ restricts to an
action of $\mathbb{R}$ on $G_{0}\widetilde{S}.$ Observe that for any two
points $\xi,\eta \in \partial \widetilde{S}$ there exists a unique (up to
parametrization) geodesic $g$ in $G_{0}\widetilde{S}$ with $g\left(
-\infty \right)  =\xi$ and $g\left(  +\infty \right)  =\eta.$ Similarly, we
identify all local geodesics in $GS$ which are closed and parallel and, thus,
obtain the corresponding space $G_{0}S.$

The following theorem establishes the above mentioned uniqueness of
(non-closed) geodesic lines in $\widetilde{S}.$

\begin{theorem}
\label{no_strips}Let $g$ be a non-closed geodesic or geodesic ray in $S.$
Then
\[
d\left(  \operatorname{Im}g,\left \{  s_{1},...,s_{n}\right \}  \right)  =0.
\]

\end{theorem}

For its proof we need the following elementary lemma.

\begin{lemma}
\label{euclidean}Let $FS\left(  \varepsilon \right)  ,$ $FS\left(
\delta \right)  $ be two infinite flat strips in $\widetilde{S}$ of width
$\varepsilon,$ $\delta$ respectively with $0<\varepsilon \leq \delta.$ Assume
that $FS\left(  \varepsilon \right)  ,$ $FS\left(  \delta \right)  $ intersect
at an angle $\theta,$ with $0<\theta<\pi/2.$ Then, $FS\left(  \varepsilon
\right)  \cup FS\left(  \delta \right)  $ contains a right angle quadrangle of
width $\delta+\frac{\varepsilon}{2\cos \theta}$ and length $\frac{\varepsilon
}{2\sin \theta}.$
\end{lemma}

\begin{proof}
The proof is elementary. Figure \ref{quardangle} exhibits the details, where
the segment
\[
AD=\frac{1}{2}FC=\frac{1}{2}\frac{\varepsilon}{\sin \theta}%
\]
is understood as the length and
\[
GA=GB+BA=\delta+\frac{EB}{\cos \theta}=\delta+\frac{\varepsilon}{2\cos \theta
}>\delta+\frac{\varepsilon}{2}%
\]
the width.
\end{proof}

\begin{remark}
\label{euclidean_remark}Clearly, by letting $\theta \rightarrow0$ the length of
the (finite) flat quadrangle given by the above Lemma tends to $\infty,$ while
its width is bounded below by $\delta+\frac{\varepsilon}{2}.$
\end{remark}

\begin{lemma}
\label{doublep}Let $\widetilde{S}$ be the universal cover of e.s.c.s $S.$ Let
$\left[  x_{0},y_{0}\right]  $ be a geodesic segment containing a conical
point $s$ in its interior. Assume that both angles subtended by $\left[
x_{0},s\right]  $ and $\left[  s,y_{0}\right]  $ at $s$ (notation
$\measuredangle_{s}\left(  x_{0},y_{0}\right)  $) are strictly bigger than
$\pi,$ say, $\theta_{1}+\pi$ and $\theta_{2}+\pi$ with $\theta_{1}+\theta
_{2}+2\pi=\theta \left(  s\right)  .$ Then there exists an $\varepsilon>0$ such
that
\[
\forall x,y\mathrm{\  \ with\  \ }d\left(  x,x_{0}\right)  <\varepsilon
\mathrm{\  \ and\  \ }d\left(  y,y_{0}\right)  <\varepsilon \Rightarrow \left[
x,y\right]  \ni s.
\]

\end{lemma}

\begin{proof}
Set $\theta_{0}=\frac{1}{3}\min \left \{  \theta_{1},\theta_{2},\pi \right \}  .$
We may choose a neighborhood around $x_{0}$ of radius $\varepsilon>0$ such
that the angle $\measuredangle_{s}\left(  x_{0},x\right)  <\theta_{0}$ for all
$x$ with $d\left(  x,x_{0}\right)  <\varepsilon.$ Similarly for $y_{0}.$ Then
for any $x,y$ with $d\left(  x,x_{0}\right)  ,d\left(  y,y_{0}\right)
<\varepsilon$ we have that both angles subtended by $\left[  x,s\right]  $ and
$\left[  s,y\right]  $ at $s$ are $>\pi.$ Thus $\left[  x,s\right]
\cup \left[  s,y\right]  $ is a geodesic segment.
\end{proof}

\begin{lemma}
\label{epsilon}Assume that a sequence of geodesic rays $\left \{
g_{n}\right \}  $ converges to a geodesic ray $g.$ If for some $\varepsilon>0,$
$g_{n}\in FHS\left(  \varepsilon \right)  $ for all $n,$ then $g\in FHS\left(
\varepsilon \right)  .$\newline The same result holds for a sequence of
geodesic segments converging to a geodesic ray, provided that the lengths of
the segments tends to $\infty.$
\end{lemma}

\begin{proof}
We write the proof for geodesic rays as the proof for segments is
identical.\newline Assume that $\varepsilon \left(  g\right)  <\varepsilon.$
Note that $\varepsilon \left(  g\right)  $ may be $0.$ We first treat the case
where $g$ contains a conical singularity. In this case, we have that, for some
lift $\widetilde{g}$ of $g$ and pre-image $\widetilde{s}$ of a conical
singularity $s\in \left \{  s_{1},...,s_{n}\right \}  ,$ $\widetilde{g}\left(
t_{s}\right)  =\widetilde{s}$ for some $t_{s}\in \left(  0,\infty \right)  .$
There are two angles formed by the segments of $\widetilde{g}$ at
$\widetilde{g}\left(  t_{s}\right)  .$ One of them must be equal to $\pi,$
otherwise, by Lemma \ref{doublep} and large enough $N$ we have that the image
of $\widetilde{g_{N}}$ contains $\widetilde{s},$ a contradiction. Without loss
of generality, assume that the angle on the left, according to the
parametrization of the ray $\widetilde{g},$ is equal to $\pi.$

If $\operatorname{Im}\widetilde{g}$ contains another pre-image $\widetilde
{s}^{\prime}\neq \widetilde{s}$ of a conical singularity $s^{\prime}\in \left \{
s_{1},...,s_{n}\right \}  ,$ that is, $\widetilde{g}\left(  t_{s^{\prime}%
}\right)  =\widetilde{s}^{\prime}$ for some $t_{s^{\prime}}\in \left(
0,\infty \right)  ,$ then, for the same reason as above, one of the two angles
formed by the segments of $\widetilde{g}$ at $\widetilde{g}\left(
t_{s^{\prime}}\right)  $ must be equal to $\pi.$

\noindent \textit{Claim:} The angle subtended at $\widetilde{g}\left(
t_{s^{\prime}}\right)  $ which is equal to $\pi$ is also on the left,
according to the parametrization, of the ray $\widetilde{g}.$

\noindent \textit{Proof of Claim:} Assume the Claim does not hold and without
loss of generality assume $t_{s^{\prime}}>t_{s}.$ Pick $N\in \mathbb{N}$ such
that%
\[
d\left(  \widetilde{g}\left(  t\right)  ,\widetilde{g}_{N}\left(  t\right)
\right)  <\varepsilon/4\mathrm{\  \ for\  \ all\  \ }t\in \left[  0,2t_{s^{\prime
}}\right]  .
\]
Then, $\widetilde{g}_{N}\left(  t_{s}\right)  $ is on the left of
$\operatorname{Im}\widetilde{g}$ and $\widetilde{g}_{N}\left(  t_{s^{\prime}%
}\right)  =\widetilde{s}^{\prime}$ is on the right of $\operatorname{Im}%
\widetilde{g}.$ Thus, the singularity $\widetilde{s}$ lies on the right of
$\operatorname{Im}\widetilde{g}_{N}$ at distance $<\varepsilon/2$ and
$\widetilde{s}^{\prime}$ lies on the left of $\operatorname{Im}\widetilde
{g}_{N}$ at distance $<\varepsilon/2.$ This is a contradiction since
$\widetilde{g}_{N}$ is contained in $FHS\left(  \varepsilon \right)  .$ This
completes the proof of the Claim.

Using the Claim we have that the assumption $\varepsilon \left(  g\right)
<\varepsilon$ implies that there must exist a conical singularity
$\widetilde{s_{1}}\in \widetilde{S}$ on the left of the ray $\widetilde{g}$
such that
\[
0<d\left(  \widetilde{s_{1}},\operatorname{Im}\widetilde{g}\right)
=\varepsilon_{1}^{\prime}<\varepsilon.
\]
Otherwise, we would be able to find a flat half strip containing
$\widetilde{g}$ of width $\varepsilon^{\prime}>\varepsilon \left(  g\right)  .$
Let $\left[  \widetilde{s_{1}},\widetilde{g}\left(  t_{1}\right)  \right]  $
$,t_{1}\in$ $\left[  0,\infty \right)  $ be the geodesic segment of length
$\varepsilon_{1}^{\prime}$ realizing the distance of $\widetilde{s_{1}}$ from
$\operatorname{Im}\widetilde{g}.$ For simplicity we assume that $t_{1}>t_{s}.$
The other case is treated similarly.

Each geodesic $g_{n}$ splits the flat half strip $FHS\left(  \varepsilon
\right)  $ to which it is contained into two flat half strips whose
intersection is $\operatorname{Im}\widetilde{g}_{n}.$ Denote by $\varepsilon
_{1,n}$ and $\varepsilon_{2,n}$ the widths of these flat half strips with
$\varepsilon_{1,n}+\varepsilon_{2,n}=\varepsilon.$ Clearly, we may assume that
$\widetilde{s_{1}}$ is on the left of the ray $\widetilde{g_{n}}$ and,
moreover, by Lemma \ref{doublep} it follows that for large enough $n$ the
conical singularity $\widetilde{s}=$ $\widetilde{g}\left(  t_{s}\right)  $ is
on the right of the ray $\widetilde{g_{n}},$ otherwise $\widetilde{g_{n}}$
cannot converge to $\widetilde{g}.$ Choose $N\in \mathbb{N}$ such that
\begin{equation}
\forall t\in \left[  0,t_{1}\right]  ,d\left(  \widetilde{g}_{N}\left(
t\right)  ,\widetilde{g}\left(  t\right)  \right)  <\frac{\varepsilon
-\varepsilon_{1}^{\prime}}{3}. \label{contra1}%
\end{equation}
Since $\widetilde{s}$ does not belong to the flat half strip containing
$\operatorname{Im}\widetilde{g_{N}},$ we have $\varepsilon_{2,N}<d\left(
\widetilde{g_{N}}\left(  t_{s}\right)  ,\widetilde{s}\right)  ,$ thus
\begin{equation}
\varepsilon_{2,N}<d\left(  \widetilde{g_{N}}\left(  t_{s}\right)
,\widetilde{g}\left(  t_{s}\right)  \right)  <\frac{\varepsilon-\varepsilon
_{1}^{\prime}}{3}. \label{contra2}%
\end{equation}
Similarly,
\[
\varepsilon_{1,N}<d\left(  \widetilde{g_{N}}\left(  t_{1}\right)
,\widetilde{s_{1}}\right)  \leq d\left(  \widetilde{g_{N}}\left(
t_{1}\right)  ,\widetilde{g}\left(  t_{1}\right)  \right)  +d\left(
\widetilde{g}\left(  t_{1}\right)  ,\widetilde{s_{1}}\right)
\]
which implies that
\begin{equation}
\varepsilon_{1,N}<\frac{\varepsilon-\varepsilon_{1}^{\prime}}{3}%
+\varepsilon_{1}^{\prime}. \label{contra3}%
\end{equation}
Combining inequalities \ref{contra2},\ref{contra3} we obtain%
\[
\varepsilon=\varepsilon_{1,N}+\varepsilon_{2,N}<\frac{2}{3}\left(
\varepsilon-\varepsilon_{1}^{\prime}\right)  +\varepsilon_{1}^{\prime}%
\]
a contradiction, which completes the proof in the case $g$ contains a conical singularity.

Assume now that $g$ does not contain a conical singularity. Then the
assumption $\varepsilon \left(  g\right)  <\varepsilon$ implies that for some
(hence any) lift $\widetilde{g}$ of $g,$ there must exist conical
singularities $\widetilde{s_{1}},\widetilde{s_{2}}\in \widetilde{S}$ such that

\begin{itemize}
\item $d\left(  \widetilde{s_{1}},\operatorname{Im}\widetilde{g}\right)
=\varepsilon_{1}^{\prime}>0,$ $d\left(  \widetilde{s_{2}},\operatorname{Im}%
\widetilde{g}\right)  =\varepsilon_{2}^{\prime}>0,$

\item $0<\varepsilon_{1}^{\prime}+\varepsilon_{2}^{\prime}<\varepsilon,$ and

\item $\widetilde{s_{1}},\widetilde{s_{2}}$ do not belong to the same side of
$\operatorname{Im}\widetilde{g}.$
\end{itemize}

Otherwise, we would be able to extend the flat half strip containing
$\widetilde{g}$ to a flat half strip of width $\varepsilon^{\prime
}>\varepsilon \left(  g\right)  .$ Let $\left[  \widetilde{s_{1}},\widetilde
{g}\left(  t_{1}\right)  \right]  $ (resp. $\left.  \left[  \widetilde{s_{2}%
},\widetilde{g}\left(  t_{2}\right)  \right]  \right)  ,$ where $t_{1}\in$
$\left[  0,\infty \right)  $ (resp. $\left.  t_{2}\in \left[  0,\infty \right)
\right)  ,$ be the geodesic segment of length $\varepsilon_{1}^{\prime}$
(resp. $\varepsilon_{2}^{\prime}$) realizing the distance of $\widetilde
{s_{1}}$ (resp. $\widetilde{s_{2}})$ from $\operatorname{Im}\widetilde{g}.$ We
may assume that $t_{2}>t_{1}.$

Choose $N\in \mathbb{N}$ such that%
\begin{equation}
\forall t\in \left[  0,t_{2}\right]  ,d\left(  \widetilde{g}_{N}\left(
t\right)  ,\widetilde{g}\left(  t\right)  \right)  <\frac{\varepsilon-\left(
\varepsilon_{1}^{\prime}+\varepsilon_{2}^{\prime}\right)  }{3}. \label{contra}%
\end{equation}
\noindent If both inequalities
\begin{equation}
d\left(  \widetilde{s_{1}},\widetilde{g_{N}}\left(  t_{1}\right)  \right)
<\varepsilon_{1,N}\mathrm{\  \ and\  \ }d\left(  \widetilde{s_{1}}%
,\widetilde{g_{N}}\left(  t_{1}\right)  \right)  <\varepsilon_{2,N}
\label{twoineq}%
\end{equation}
hold, then $\widetilde{s_{1}}$ belongs to the flat half strip $FHS\left(
\varepsilon \right)  $ containing $\widetilde{g}_{N},$ a contradiction. Thus,
at least one of the inequalities in \ref{twoineq} does not hold. Without loss
of generality we may assume that the inequality
\[
d\left(  \widetilde{s_{1}},\widetilde{g_{N}}\left(  t_{1}\right)  \right)
>\varepsilon_{1,N}%
\]
holds and, therefore, the inequality
\begin{equation}
\varepsilon_{1}^{\prime}+d\left(  \widetilde{g}\left(  t_{1}\right)
,\widetilde{g_{N}}\left(  t_{1}\right)  \right)  =d\left(  \widetilde{s_{1}%
},\widetilde{g}\left(  t_{1}\right)  \right)  +d\left(  \widetilde{g}\left(
t_{1}\right)  ,\widetilde{g_{N}}\left(  t_{1}\right)  \right)  >\varepsilon
_{1,N} \label{1stfinal}%
\end{equation}
holds. Similarly, and using the assumption that $\widetilde{s_{1}}%
,\widetilde{s_{2}}$ do not belong to the same side of $\operatorname{Im}%
\widetilde{g}$ we have
\begin{equation}
\varepsilon_{2}^{\prime}+d\left(  \widetilde{g}\left(  t_{2}\right)
,\widetilde{g_{N}}\left(  t_{2}\right)  \right)  =d\left(  \widetilde{s_{2}%
},\widetilde{g}\left(  t_{2}\right)  \right)  +d\left(  \widetilde{g}\left(
t_{2}\right)  ,\widetilde{g_{N}}\left(  t_{2}\right)  \right)  >\varepsilon
_{2,N} \label{2ndfinal}%
\end{equation}
Combining the above inequalities we reach
\[
\varepsilon_{1}^{\prime}+\varepsilon_{2}^{\prime}+d\left(  \widetilde
{g}\left(  t_{1}\right)  ,\widetilde{g_{N}}\left(  t_{1}\right)  \right)
+d\left(  \widetilde{g}\left(  t_{2}\right)  ,\widetilde{g_{N}}\left(
t_{2}\right)  \right)  >\varepsilon_{1,N}+\varepsilon_{2,N}=\varepsilon
\]
which implies that
\[
2\max \left \{  d\left(  \widetilde{g}\left(  t_{1}\right)  ,\widetilde{g_{N}%
}\left(  t_{1}\right)  \right)  ,d\left(  \widetilde{g}\left(  t_{2}\right)
,\widetilde{g_{N}}\left(  t_{2}\right)  \right)  \right \}  >\varepsilon
-\left(  \varepsilon_{1}^{\prime}+\varepsilon_{2}^{\prime}\right)
\]
Thus for some $t\in \left \{  t_{1},t_{2}\right \}  $ we have that
\[
d\left(  \widetilde{g}_{N}\left(  t\right)  ,\widetilde{g}\left(  t\right)
\right)  >\frac{\varepsilon-\left(  \varepsilon_{1}^{\prime}+\varepsilon
_{2}^{\prime}\right)  }{2}%
\]
a contradiction by (\ref{contra}).
\end{proof}

\vspace{2mm}

For the proof of Theorem \ref{no_strips} we also need the following
proposition which is of interest in its own right.

\begin{proposition}
\label{u_ubar}Let $g$ be a geodesic which does not contain a conical point,
$\mathcal{U}=\left \{  t\cdot g|t\in \mathbb{R}\right \}  $ and $\overline
{\mathcal{U}}$ its closure in $GS$ under the compact open topology. Then the
following are equivalent: \newline(a) $g$ is not periodic.\newline(b)
$\overline{\mathcal{U}}\setminus \mathcal{U}\neq \emptyset,$ that is,
$\mathcal{U}$ is not closed.
\end{proposition}

\begin{proof}
Clearly, if $g$ is periodic then $\mathcal{U}$ is closed. For the converse,
assume, on the contrary, that for a non-periodic $g$ we have $\overline
{\mathcal{U}}=\mathcal{U}$. Pick a sequence $\left \{  t_{n}\right \}  $
converging to $+\infty.$ By Proposition \ref{epitelous} $GS$ is (sequentially)
compact, thus, we have, up to a subsequence, that the sequence $\left \{
t_{n}\cdot g\right \}  $ converges to a point in $\mathcal{U}$, say,%
\[
t_{n}\cdot g\rightarrow t_{0}\cdot g
\]
for some $t_{0}\in \mathbb{R}.$ Clearly, by the definition of the compact open
topology, for any $C\in \mathbb{R},$ we have
\[
\left(  t_{n}+C\right)  \cdot g\rightarrow \left(  t_{0}+C\right)  \cdot g
\]
In other words, we have the following property:

\begin{description}
\item[($*$)] every point $t\cdot g$ in $\mathcal{U}$ is the limit of a
sequence of geodesics $\left \{  t_{n}\cdot g\right \}  \subset \mathcal{U}$ for
some sequence $t_{n}\rightarrow \infty.$
\end{description}

\noindent Since $\operatorname{Im}g$ does not contain a conical point, we may
pick an open disk $D$ such that
\[
\overline{D}\mathrm{\ is\ embedded\ in\ }S,D\cap \operatorname{Im}%
g\neq \emptyset \mathrm{\ and\ }\overline{D}\cap \left \{  s_{1},...,s_{n}%
\right \}  =\emptyset.
\]
Then $\operatorname{Im}g\cap \overline{D}$ consists of geodesic segments
$\sigma_{j},j\in J$ of the form $g|_{\left[  t_{j}^{-},t_{j}^{+}\right]  }$
for some $t_{j}^{-}<t_{j}^{+}\in \mathbb{R}$ with endpoints $\sigma_{j}\left(
-\right)  =g\left(  t_{j}^{-}\right)  $ and $\sigma_{j}\left(  +\right)
=g\left(  t_{j}^{+}\right)  $ contained in $\partial D.$ The degenerate cases
$t_{j}^{-}=t_{j}^{+}$ where the segment $\sigma_{j}$ is just one point are
excluded from the collection. As the lengths of the non-null homotopic loops
in $S$ are bounded away from $0,$ the set
\[
\left \{  \left \vert t_{j}^{+}-t_{j^{\prime}}^{+}\right \vert \biggm \vert j\neq
j^{\prime}\in J\right \}
\]
is bounded below which implies that the set $\left \{  t_{j}^{+}|j\in
J\right \}  $ is a discrete subset of $\mathbb{R},$ hence $J$ is countable.
Thus, we may enumerate the segments $\sigma_{j}$ by considering $J\subseteq
\mathbb{Z}$ and $t_{j}^{-}<t_{j}^{+}<t_{j+1}^{-}<t_{j+1}^{+}$ for all $j.$

Without loss of generality, we may assume that $g\left(  0\right)  \in
\sigma_{0}\subset \operatorname{Im}g\cap \overline{D}.$ By ($\ast$), there
exists a sequence $\left \{  t_{m}\right \}  _{m\in \mathbb{R}}$ with
$t_{m}\rightarrow \infty$ such that $t_{m}\cdot g\rightarrow g.$ In particular,
there exists a sequence of geodesic segments $\sigma_{j_{m}}\backepsilon
g\left(  t_{m}\right)  $ approximating $\sigma_{0}.$ Thus, $J$ is an infinite set.

We next show that the corresponding geodesic segments $\left \{  \sigma
_{j}|j\in J\right \}  $ are infinitely many. For, if $\sigma_{j}\equiv
\sigma_{j^{\prime}}$ for some $j<j^{\prime}$ then, as $\operatorname{Im}g$
does not contain conical points, $g$ must be periodic with period
$t_{j^{\prime}}^{+}-t_{j}^{+}.$ This shows that $J$ is a countably infinite
set and so is the set
\[
\partial^{2}J:=\left \{  \left(  \sigma_{j}\left(  -\right)  ,\sigma_{j}\left(
+\right)  \right)  |j\in J\right \}  .
\]

By property ($\ast$) each geodesic segment $\sigma_{j}$ is the limit of a
sequence of segments in $\operatorname{Im}g\cap \overline{D}=\left \{
\sigma_{j},j\in J\right \}  .$ It follows that the set $\partial^{2}J$ is a
perfect subset of $\partial D\times \partial D,$ thus, uncountable, a contradiction.
\end{proof}

Using standard compactness arguments and the natural projection $GS\rightarrow
S$ given by $\gamma \rightarrow \gamma(0)$ it can be easily seen that
$\mathcal{U}\subset GS$ is not closed if and only if $\operatorname{Im}%
g\subset S$ is not closed. Consequently, we have the following

\begin{corollary}
Let $g$ be a geodesic which does not contain a conical point. Then $g$ is
closed if and only if $\operatorname{Im}g$ is a closed set.
\end{corollary}

\textbf{Proof of Theorem \ref{no_strips}} Assume, on the contrary, that
\[
d\left(  \operatorname{Im}g,\left \{  s_{1},...,s_{n}\right \}  \right)  >0.
\]
Then the maximal flat half strip $FHS\left(  \widetilde{g}\right)
\subset \widetilde{S}$ containing a lift $\widetilde{g}$ of $g$ has width
$\varepsilon \left(  g\right)  =\varepsilon>0.$ Let $\mathcal{U}=\left \{
t\cdot g|t\in \left[  0,\infty \right)  \right \}  $ and consider its closure
$\overline{\mathcal{U}}$ under the compact open topology.

Set
\[
\overline{\varepsilon}=\sup \left \{  \varepsilon \left(  h\right)
|h\in \overline{\mathcal{U}}\right \}  .
\]
By Lemma \ref{epsilon}, $\overline{\varepsilon}\geq \varepsilon.$ By
Proposition \ref{u_ubar} we may choose $h_{0}\in \overline{\mathcal{U}%
}\setminus \mathcal{U}$ such that $\varepsilon \left(  h_{0}\right)
>\overline{\varepsilon}-\frac{\varepsilon}{3}$ and let $\left \{
h_{n}\right \}  _{n\in \mathbb{N}}$ be a sequence in $\mathcal{U}$ such that
$h_{n}=t_{n}\cdot g\rightarrow h_{0}.$

Pick a lift $\widetilde{h}_{0}$ of $h_{0}$ and a sequence $\left \{
\widetilde{h}_{n}\right \}  $ approximating $\widetilde{h}_{0}$ in the compact
open topology. We first claim that for all $n\in \mathbb{N}$ large enough the
geodesic segment $\widetilde{h}_{n}\left(  \left[  \overline{\varepsilon
},T\right]  \right)  ,$ for any time $T>\overline{\varepsilon},$ cannot be
parallel to the geodesic line $\operatorname{Im}\widetilde{h}_{0}.$ For, if
$\widetilde{h}_{k}\left(  \left[  \overline{\varepsilon},T\right]  \right)  $
is parallel to the geodesic line $\operatorname{Im}\widetilde{h}_{0},$ then
$\operatorname{Im}\widetilde{h}_{0}$ would have to be parallel to the whole
geodesic line $\operatorname{Im}\widetilde{h}_{k}$ at distance, say $M\geq0.$
The same will also be true for any translation $t\cdot \widetilde{h}_{k}.$ If
$M=0$ then $\widetilde{h}_{k}$ is a translation of $\widetilde{h}_{0}$ which
implies that $h_{0}\in \mathcal{U}$, is a contradiction since $h_{0}$ was
chosen in $\overline{\mathcal{U}}\setminus \mathcal{U}$. In the case $M>0,$ as
$h_{n},$ $n>k$ is a translation of $h_{k},$ namely $h_{n}=\left(  t_{n}%
-t_{k}\right)  \cdot h_{k},$ we would have that $h_{n}\left(  0\right)  =$
$h_{k}\left(  t_{n}-t_{k}\right)  $ is at distance at least $M>0$ from
$h_{0}\left(  0\right)  .$ This is a contradiction to the fact that
$h_{n}=t_{n}\cdot g\rightarrow h_{0}.$

Pick a sequence $z_{k}\rightarrow \infty$ and for each $k\in \mathbb{N},$ let
$N\left(  k\right)  \in \mathbb{N}$ be large enough so that for the time
interval $\left[  \overline{\varepsilon},z_{k}\right]  $ the geodesic segment
$\widetilde{h}_{N\left(  k\right)  }\left(  \left[  \overline{\varepsilon
},z_{k}\right]  \right)  $ lies inside the flat half strip $FHS\left(
\widetilde{h}_{0}\right)  .$ As the geodesic segment $\widetilde{h}_{N\left(
k\right)  }\left(  \left[  \overline{\varepsilon},z_{k}\right]  \right)  $ is
not parallel to the geodesic line $\operatorname{Im}\widetilde{h}_{0},$ we may
apply Lemma \ref{euclidean} to the flat strips $FHS\left(  \widetilde{h}%
_{0}\right)  $ and $FHS\left(  \widetilde{h}_{N\left(  k\right)  }\right)
,k\in \mathbb{N}$ to obtain sequences
\[
\left \{  r_{k}\right \}  _{k\in \mathbb{N}},\left \{  q_{k}\right \}
_{k\in \mathbb{N}}\subset \left[  0,\infty \right)
\]
such that

\begin{itemize}
\item $q_{k}\rightarrow \infty,$ and

\item the geodesic segment $r_{k}\cdot \widetilde{h}_{0}\left(  \left[
0,q_{k}\right]  \right)  $ is contained in a right angle quadrangle of width
$\varepsilon \left(  \widetilde{h}_{0}\right)  +\frac{\varepsilon}{2}$ and
length $q_{k}.$
\end{itemize}

The sequence $\left \{  p_{S}\left(  r_{k}\cdot \widetilde{h}_{0}\right)
\right \}  $ converges, up to a subsequence, to a geodesic $f\in \overline
{\mathcal{U}}.$ By Lemma \ref{epsilon} and Remark \ref{euclidean_remark}, it
follows that $f\in FHS\left(  \varepsilon \left(  \widetilde{h}_{0}\right)
+\frac{\varepsilon}{2}\right)  $ which is a contradiction because
\[
\varepsilon \left(  \widetilde{h}_{0}\right)  +\frac{\varepsilon}%
{2}=\varepsilon \left(  h_{0}\right)  +\frac{\varepsilon}{2}>\varepsilon \left(
h_{0}\right)  +\frac{\varepsilon}{3}>\overline{\varepsilon}%
\]
and $\overline{\varepsilon}$ is chosen to be $\sup \left \{  \varepsilon \left(
h\right)  |h\in \overline{\mathcal{U}}\right \}  .$ \hfill \rule{0.5em}{0.5em}

\begin{remark}
As the image of any geodesic in $\widetilde{S}$ splits $\widetilde{S}$ into
two convex sets, say $\widetilde{S}_{L}$ and $\widetilde{S}_{R}$, we may speak
of the distance to the left of $\operatorname{Im}\widetilde{g}$ from the set
$\left \{  \widetilde{s}\in \widetilde{S}_{L}\bigm \vert \widetilde
{s}\mathrm{\mathrm{\ conical\mathrm{\mathrm{\ point}}}}\right \}  ,$ namely,
\[
d_{L}=\inf \left \{  d\left(  x,\widetilde{s}\right)  \bigm \vert x\in
\operatorname{Im}\widetilde{g},\widetilde{s}\in \widetilde{S}_{L}\right \}
\]
and similarly for $d_{R}.$ For a geodesic $\widetilde{g}$ in $\widetilde{S}$
which\newline(a) does not contain any conical point\newline(b) projects to a
non-closed geodesic in $S$\newline Theorem \ref{no_strips} asserts that
$d_{L}=d_{R}=0.$ This is clear in the above proof by working with a maximal
flat strip to the left (right) of $\operatorname{Im}g.$
\end{remark}

Using the uniqueness of the non-closed geodesic lines in $\widetilde{S}$ (see
Theorem \ref{no_strips}), one can obtain a short proof of the following
theorem shown in \cite{[Bal]} in the context of Hadamard spaces.

\begin{theorem}
\label{orbit}There exists a geodesic $\gamma$ in $GS$ whose orbit
$\mathbb{R}\gamma$ under the geodesic flow is dense in $GS.$
\end{theorem}

\section{Density of closed geodesics}

In this section we show the following

\begin{theorem}
\label{dense} The closed geodesics are dense in $GS.$
\end{theorem}

Set $\partial^{2}\widetilde{S}=\{(\xi,\eta)\in \partial \widetilde{S}%
\times \partial \widetilde{S}:\xi \neq \eta \}$ and let $H:G_{0}\widetilde
{S}\rightarrow \partial^{2}\widetilde{S}\times \mathbb{R}$ be the homeomorphism
given as follows: choose a base point $x_{0}\in \widetilde{S}.$ Let $\gamma
_{E}\in G_{0}\widetilde{S}.$ If the class $\gamma_{E}$ consists of a single
geodesic line $\gamma$ joining $\gamma \left(  +\infty \right)  ,$
$\gamma \left(  +\infty \right)  \in \partial \widetilde{S}$ set%

\[
H(\gamma_{E})=(\gamma(-\infty),\gamma(\infty),s)\text{ \qquad(1)}%
\]

\noindent where $s$ is the real number such that $d(x_{0},\gamma
(\mathbb{R}))=d(x_{0},\gamma(-s)).$ If $\gamma_{E}$ is a class of parallel
geodesics $\left \{  \gamma^{r}|r\in \left[  0,\varepsilon \right]  \right \}  $
forming a flat strip $E,$ then there exist a unique geodesic segment starting
from $x_{0}$ and realizing the distance of $x_{0}$ from $E.$ We may extend
this geodesic segment to a geodesic segment which is perpendicular to all
lines $\operatorname{Im}\gamma^{r},$ $r\in \left[  0,\varepsilon \right]  .$ In
other words, there exists a unique $s\in \mathbb{R}$ such that
\[
d\left(  x_{0},\gamma^{r}(\mathbb{R})\right)  =d(x_{0},\gamma^{r}(-s)),\forall
r\in \left[  0,\varepsilon \right]  .
\]
We use this $s\in \mathbb{R}$ to define $H(\gamma_{E})$ as in equation (1). The
map $H$ is a homeomorphism (see \cite[Th. 4.8]{[Cha]}).\\[2mm]\noindent
\textbf{Proof of Theorem \ref{dense}} Let $\beta \in GS$ be a non-closed
geodesic and $\widetilde{\beta}\in G\widetilde{S}$ a lifting. Set
$\eta=\widetilde{\beta}(-\infty),$ $\xi=\widetilde{\beta}(+\infty).$ Then by
Proposition \ref{dense3} there exists a sequence $\left \{  \phi_{n}\right \}  $
of hyperbolic isometries such that $\phi_{n}\left(  -\infty \right)
\rightarrow \eta$ and $\phi_{n}\left(  +\infty \right)  \rightarrow \xi.$ Set
$\eta_{n}\equiv$ $\phi_{n}\left(  -\infty \right)  $ and $\xi_{n}\equiv \phi
_{n}\left(  +\infty \right)  $ and, clearly, $\phi_{n}(\eta_{n})=\eta_{n},$
$\phi_{n}(\xi_{n})=\xi_{n}.$ Since $\partial \widetilde{S}$ is topologically a
circle, we may choose the sequence $\{(\eta_{n},\xi_{n})\} \subset
\partial_{vis}\widetilde{S}\times \partial_{vis}\widetilde{S}$ so that, in
addition, for all $n$ the following condition holds:

\begin{quote}
$\eta_{n},\xi_{n}$ belong to the same component of $\partial \widetilde
{S}\setminus \left \{  \eta,\xi \right \}  .$
\end{quote}

As $\beta$ is non-closed, the chosen lift $\widetilde{\beta}\in G\widetilde
{S}$ determines a class $\widetilde{\beta}_{E}\in G_{0}\widetilde{S}$ which is
a singleton. We have $H(\widetilde{\beta}_{E})=(\eta,\xi,s)$ for some
$s\in \mathbb{R}$. Denote by $\widetilde{\beta_{n}}_{E}$ the geodesic
$H^{-1}(\eta_{n},\xi_{n},s_{n})\in G_{0}\widetilde{S}$ where $s_{n}$ is chosen
so that $s_{n}\rightarrow s.$ As $H$ is a homeomorphism, $\widetilde{\beta
_{n}}_{E}\rightarrow \widetilde{\beta}_{E}$ in the compact-open topology of
$G_{0}\widetilde{S}$. If the class $\widetilde{\beta_{n}}_{E}$ is not a
singleton, choose $\widetilde{\beta_{n}}$ to be $\widetilde{\beta_{n}}^{r}$
where $r=0$ or $1$ according to which of the two geodesic lines
$\operatorname{Im}\widetilde{\beta_{n}}^{0},$ $\operatorname{Im}%
\widetilde{\beta_{n}}^{1}$ has smaller distance (as a set) from
$\operatorname{Im}\widetilde{\beta_{n}}.$ We then have that the sequence
$\left \{  \widetilde{\beta_{n}}_{E}\right \}  $ determines a sequence $\left \{
\widetilde{\beta_{n}}\right \}  $ such that $\widetilde{\beta_{n}}%
\rightarrow \widetilde{\beta}$ in the compact-open topology of $G\widetilde{S}%
$. Clearly, each $\widetilde{\beta_{n}}$ is translated by $\phi_{n}$ and,
thus, projects to a closed geodesic $\beta_{n}=\pi_{S}(\widetilde{\beta_{n}})$
in $S.$ By continuity of $\pi_{S},$ we have the desired convergence $\beta
_{n}\rightarrow \beta.$ \hfill \rule{0.5em}{0.5em}

\section{Application to surfaces with conical singularities $<2\pi.$}

From now on we consider surfaces with conical singularities of arbitrary angle
$\theta \left(  s\right)  \in \left(  0,\infty \right)  \setminus \left \{
2\pi \right \}  .$ We will show that any geodesic in $S$ can be approximated, in
the compact open topology, by closed geodesics and/or appropriate saddle connections.

In what follows we will need a certain existence theorem of branched coverings
which follows from Theorem 2.1 in \cite{[PP]}. The latter is the result of
efforts of many mathematicians. Originally, the classical problem dating back
to Hurewitz asks necessary and sufficient conditions for the existence of
branched covering (see in \cite{[PP]} and its bibliography for supplementary
information). We next recall basic definitions.

A branched covering is a map $\psi:S^{\prime}\rightarrow S$, where $S^{\prime
}$ and $S$ are closed connected surfaces and $\psi$ is locally modeled on maps
of the form $\mathbb{C}\ni z \rightarrow z^{k} \in \mathbb{C}$ for some $k
\geq1.$ The integer $k$ is called the local degree at the point of $S^{\prime
}$ corresponding to $0$ in the source $\mathbb{C}.$ If $k > 1$ then the point
of $S$ corresponding to $0$ in the target $\mathbb{C}$ is called a branching
point. The branching points are isolated, hence there are finitely many, say
$n,$ of them. Removing the branching points in $S$ and all their pre-images in
$S^{\prime}$, the restriction of $\psi$ gives a genuine covering, whose degree
we will denote by $d.$ If the i-th branching point on $S$ has $m_{i}$
pre-images, the local degrees $(d_{ij})_{j=1,\ldots.m_{i}}$ at these points
give a partition of $d$ i.e., $d_{ij}\geq1$ and $\sum_{j=1}^{m_{i}} d_{ij}=d.$

In our context, we consider only the conical points $s_{1},\ldots,s_{l},$
$l\leq n$ in $S$ which are conical singularities of angle smaller than $2\pi$
and choose

\begin{itemize}
\item $m_{i}=1$ for all $i=1,\ldots,n$

\item $d_{i1}=d$ odd natural number so that $d_{i1}\theta(s_{i})>2\pi.$
\end{itemize}

The above branching data is realizable by a branched covering (see Theorem 2.1
in \cite{[PP]}) yielding the following theorem

\begin{theorem}
\label{branched covering}For every $e.s.c.s.$ $S$ of genus $\geq0,$ there a
(finite) branched covering $\psi:S^{\prime}\rightarrow S$ such that:\newline%
(1) $S^{\prime}$ is an $e.s.c.s.,$ with all conical singularities being of
angle larger than $2\pi,$ \newline(2) the branch set on $S$ are the conical
singularities of angle smaller than $2\pi.$\newline In particular, each
conical singularity of angle smaller than $2\pi$ has one pre-image in
$S^{\prime}.$
\end{theorem}

Let $S$ be a $e.s.c.s.$ with conical singularities $\left \{  s_{1}%
,\ldots,s_{l},s_{l+1},\ldots,s_{n}\right \}  $ with $\theta \left(
s_{i}\right)  \in \left(  0,2\pi \right)  $ for $i=1,\ldots,l$ and
$\theta \left(  s_{i}\right)  \in \left(  2\pi,\infty \right)  $ for
$i=l+1,\ldots,n,$ $l\leq n.$ Observe that the image of a geodesic in such
surfaces, cannot contain a singular point $s$ with $\theta \left(  s\right)
\in \left(  0,2\pi \right)  .$

We consider the following notion in $S:$ a local geodesic segment in $S$ with
endpoints in $\left \{  s_{1},\ldots,s_{l}\right \}  $ will be called a
\emph{generalized saddle connection}. This terminology follows the notion of a
saddle connection originally introduced in the study of translation surfaces
(see \cite{Zor}). Observe that a generalized saddle connection can only
contain in its interior conical points $s_{i},$ for $i=l+1,\ldots,n.$ We also
allow the endpoints to coincide. We will also use the notion of a \emph{closed
piece-wise geodesic} by which we mean a finite union of generalized saddle
connections which form a closed curve. Clearly, a closed piece-wise geodesic
is a closed curve $\gamma$ which is a local geodesic except at the points
conical points $\operatorname{Im}\gamma \cap$ $\left \{  s_{1},\ldots
,s_{l}\right \}  .$

\begin{theorem}
\label{dense_small}Let $S$ be a Euclidean surface with conical singularities
\[
\left \{  s_{1},\ldots,s_{l},s_{l+1},\ldots,s_{n}\right \}
\]
with $\theta \left(  s_{i}\right)  \in \left(  0,2\pi \right)  $ for
$i=1,\ldots,l$ and $\theta \left(  s_{i}\right)  \in \left(  2\pi,\infty \right)
$ for $i=l+1,\ldots,n,$ $l\leq n.$ Then,\newline(a) for any non-closed
geodesic or geodesic ray $g$ in $S,$
\[
d\left(  \operatorname{Im}g,\left \{  s_{1},\ldots,s_{l},s_{l+1},\ldots
,s_{n}\right \}  \right)  =0.
\]
(b) every element in $GS$ can be approximated, in the compact open topology,
either by a sequence of closed geodesics or, by a sequence of closed
piece-wise geodesics.
\end{theorem}

\begin{proof}
Let $S^{\prime}$ be the $e.s.c.s.$ and $\psi:S^{\prime}\rightarrow S$ the
branched covering posited by the above theorem. The branched set of $\psi$ is
$\left \{  s_{1},\ldots,s_{l}\right \}  $ and denote by $\left \{  s_{1}^{\prime
},\ldots,s_{l}^{\prime}\right \}  $ their pre-images in $S^{\prime}.$

Let $g$ be a non-closed geodesic (or geodesic ray). Observe that
$\operatorname{Im}g$ does not contain any of the singularities $s_{1}%
,\ldots,s_{l}.$ Clearly, $g$ is covered by finitely many geodesics in
$GS^{\prime}.$ Let $g^{\prime}$ be one of them.

For (a), if
\[
d\left(  \operatorname{Im}g,\left \{  s_{1},\ldots,s_{l}, s_{l+1},\ldots,s_{n}
\right \}  \right)  >0
\]
then, since $\psi$ is distance decreasing, we have
\[
d\left(  \operatorname{Im}g^{\prime},\left \{  s_{1}^{\prime},\ldots
,s_{l}^{\prime}, s_{l+1}^{\prime},\ldots,s_{n}^{\prime} \right \}  \right)
>0.
\]
This is a contradiction because Theorem \ref{no_strips} asserts that
\[
d\left(  \operatorname{Im}g^{\prime},\left \{  s_{1}^{\prime},\ldots
,s_{l}^{\prime},s_{l+1}^{\prime},\ldots,s_{n}^{\prime}\right \}  \right)  =0
\]
For (b), Theorem \ref{dense} asserts that there exists a sequence of closed
geodesics $\left \{  g_{n}^{\prime}\right \}  $ converging, in the compact open
topology, to $g^{\prime}.$ If
\begin{equation}
\operatorname{Im}g_{n}^{\prime}\cap \left \{  s_{1}^{\prime},\ldots
,s_{l}^{\prime}\right \}  =\varnothing \label{empty_tonos}%
\end{equation}
holds for infinitely many $n,$ then by passing, if necessary, to a subsequence
we have that $\left \{  g_{n}^{\prime}\right \}  $ converges to $g^{\prime}$
and, therefore, the sequence of closed geodesics $\left \{  \psi \left(
g_{n}^{\prime}\right)  \right \}  $ converges to $\psi \left(  g^{\prime
}\right)  =g.$

If (\ref{empty_tonos}) holds only for finitely many $n,$ then we have that
$\left \{  \psi \left(  g_{n}^{\prime}\right)  \right \}  $ is a sequence of
closed piece-wise geodesics which converges, in the compact-open topology, to
$g.$
\end{proof}

We now turn our attention to Euclidean surfaces with conical singularities
$\left \{  s_{1},\ldots,s_{l},s_{l+1},\ldots,s_{n}\right \}  $which satisfy
$\theta \left(  s_{i}\right)  \in \left(  0,\pi \right)  $ for $i=1,\ldots,l$ and
$\theta \left(  s_{i}\right)  \in \left(  2\pi,\infty \right)  $ for
$i=l+1,\ldots,n,$ $l\leq n.$ Note that a geodesic, which, as usual, is defined
to be a local isometric map, may have homotopically trivial self
intersections, that is,

\begin{center}
$\exists t_{1},t_{2}\in \mathbb{R}$ with $g\left(  t_{1}\right)  =g\left(
t_{2}\right)  $ such that the loop $g|_{\left[  t_{1},t_{2}\right]  }$ is contractible.
\end{center}

\noindent Clearly, the lift $\widetilde{g}$ to the universal cover
$\widetilde{S}$ of $S$ of a geodesic $g$ with homotopically trivial self
intersections is not a global isometric map. In view of this and the following
Lemma, we restrict our attention to geodesics which do not have homotopically
trivial self intersections, equivalently, from now on the word geodesic will
mean that its lift to the universal cover $\widetilde{S}$ is a (global) geodesic.

\begin{lemma}
There exists a positive real number $C$ such that for any geodesic $g$
\[
d\left(  g\left(  t\right)  ,s_{i}\right)  \geq
C\mathrm{\  \  \mathrm{for\  \ all\ }}i=1,\ldots,l\mathrm{\ and\ for\ all\  \ }%
t\in \mathbb{R}.
\]

\end{lemma}

\begin{proof}
It suffices to show that for each $i=1,\ldots,l$ there exists $C_{i}>0$
depending on $\theta \left(  s_{i}\right)  ,$ such that for any geodesic $g$
\[
d\left(  g\left(  t\right)  ,s_{i}\right)  \geq C_{i}%
\mathrm{\  \ for\  \ all\  \ }t\in \mathbb{R}.
\]
Choose $C_{i}^{\prime}$ such that $d\left(  s_{i},s_{j}\right)  >C_{i}%
^{\prime}$ for all $j\neq i.$ Set $C_{i}=C_{i}^{\prime}\cos \frac{\theta \left(
s_{i}\right)  }{2}.$ We will show that if $d\left(  \operatorname{Im}%
g,s_{i}\right)  <C_{i}$ then $g$ has a homotopically trivial self intersection.

Let $g\left(  t_{i}\right)  $ be the point on $\operatorname{Im}g$ of minimum
distance, say $C_{0},$ from $s_{i}.$ Clearly, $d\left(  g\left(  t_{i}\right)
,s_{i}\right)  \equiv C_{0}<C_{i}.$ Then, the geodesic segment $\left[
s_{i},g\left(  t_{i}\right)  \right]  $ is perpendicular to $\operatorname{Im}%
g.$ Let $r$ be the geodesic ray emanating from $s_{i}$ such that both angles
subtended by $r$ and $\left[  s_{i},g\left(  t_{i}\right)  \right]  $ at
$s_{i}$ are equal to $\frac{\theta \left(  s_{i}\right)  }{2}.$ Set
\[
T=\frac{C_{0}}{\cos \frac{\theta \left(  s_{i}\right)  }{2}}\mathrm{\ and\ }%
t_{i}^{\prime}=C_{0}\tan \frac{\theta \left(  s_{i}\right)  }{2}.
\]
Then the geodesic segments
\[
\left[  s_{i},g\left(  t_{i}\right)  \right]  ,\left[  g\left(  t_{i}\right)
,g\left(  t_{i}+t_{i}^{\prime}\right)  \right]  \mathrm{\ and\ }r|_{\left[
0,T\right]  }%
\]
and the geodesic segments
\[
\left[  s_{i},g\left(  t_{i}\right)  \right]  ,\left[  g\left(  t_{i}\right)
,g\left(  t_{i}-t_{i}^{\prime}\right)  \right]  \mathrm{\ and\ }r|_{\left[
0,T\right]  }%
\]
form two (equal) right triangles with common hypotenuse $r|_{\left[
0,T\right]  }$ and common side $\left[  s_{i},g\left(  t_{i}\right)  \right]
.$ Thus,
\[
g\left(  t_{i}+t_{i}^{\prime}\right)  =r\left(  T\right)  =g\left(
t_{i}-t_{i}^{\prime}\right)
\]
and $g$ has a self intersection which is clearly homotopically trivial.
\end{proof}

In the next Theorem $GS$ consists, as mentioned above, of geodesics with no
homotopically trivial self intersections.

\begin{theorem}
\label{dense_very_small}Let $S$ be a Euclidean surface with conical
singularities
\[
\left \{  s_{1},\ldots,s_{l},s_{l+1},\ldots,s_{n}\right \}
\]
with $\theta \left(  s_{i}\right)  \in \left(  0,\pi \right)  $ for
$i=1,\ldots,l$ and $\theta \left(  s_{i}\right)  \in \left(  2\pi,\infty \right)
$ for $i=l+1,\ldots,n,$ $l\leq n.$ Then, every element in $GS$ can be
approximated, in the compact open topology, either by a sequence of closed
geodesics or, by a sequence of generalized saddles.
\end{theorem}

\begin{proof}
As in the previous theorem, consider the branched covering $\psi:S^{\prime
}\rightarrow S$ with branched set $\left \{  s_{1},\ldots,s_{l}\right \}  $ with
pre-images $\left \{  s_{1}^{\prime},\ldots,s_{l}^{\prime}\right \}  $ in
$S^{\prime}.$ Let $\beta \in GS$ be a non-closed geodesic. By the previous
lemma,
\[
\exists C>0\mathrm{\  \ such\  \ that\  \ }d\left(  \beta \left(  t\right)
,s_{i}\right)  \geq C,\mathrm{\ for\  \ all\  \ }i=1,\ldots
,l,\mathrm{\  \ and\ for\  \ all\  \ }t\in \mathbb{R}.
\]
An analogous statement follows for any lift $\beta^{\prime}$ of $\beta$ and
the singularities $\left \{  s_{1}^{\prime},\ldots,s_{l}^{\prime}\right \}  $
\begin{equation}
\exists C>0\mathrm{\  \ such\  \ that\  \ }d\left(  \beta^{\prime}\left(
t\right)  ,s_{i}^{\prime}\right)  \geq C,\mathrm{\  \ for\  \ all\  \ }%
i=1,\ldots,l,\mathrm{\  \ and\  \ }t\in \mathbb{R}. \label{si}%
\end{equation}
Theorem \ref{dense} asserts that there exists a sequence of closed geodesics
$\left \{  \beta_{n}^{\prime}\right \}  $ converging, in the compact open
topology, to $\beta^{\prime}.$ If
\begin{equation}
\operatorname{Im}\beta_{n}^{\prime}\cap \left \{  s_{1}^{\prime},\ldots
,s_{l}^{\prime}\right \}  =\varnothing \label{empty_tonos_duo}%
\end{equation}
holds for infinitely many $n,$ then we obtain as in the proof of the previous
theorem that the sequence of closed geodesics $\left \{  \psi \left(  \beta
_{n}^{\prime}\right)  \right \}  $ converges to $\psi \left(  \beta^{\prime
}\right)  =\beta.$

If (\ref{empty_tonos_duo}) holds only for finitely many $n,$ we may assume
that for all $n,$ $\operatorname{Im}\beta_{n}^{\prime}$ contains at least one
singularity from $\left \{  s_{1}^{\prime},\ldots,s_{l}^{\prime}\right \}  .$
Choose a sequence of positive $\varepsilon_{k}\rightarrow0$ with
$\varepsilon_{k}<C.$ Then, for each $k\in \mathbb{N},$ there exists a closed
geodesic $\beta_{n\left(  k\right)  }^{\prime}$ which $\left(  \varepsilon
_{k},\left[  -k,k\right]  \right)  $-approximates $\beta^{\prime},$ that is

for $t\in \left[  -k,k\right]  ,$ $d\left(  \beta_{n\left(  k\right)  }%
^{\prime}\left(  t\right)  ,\beta^{\prime}\left(  t\right)  \right)
<\varepsilon_{k}$

By (\ref{si}), $\beta_{n\left(  k\right)  }^{\prime}|_{\left[  -k,k\right]  }$
does not contain a singularity $s_{i},i\leq l.$ Restrict $\beta_{n\left(
k\right)  }^{\prime}$ to a compact set $\left[  t_{-k},t_{k}\right]  $
containing $\left[  -k,k\right]  $ such that
\[
\beta_{n\left(  k\right)  }^{\prime}\left(  t_{-k}\right)  ,\beta_{n\left(
k\right)  }^{\prime}\left(  t_{k}\right)  \in \left \{  s_{1}^{\prime}%
,\ldots,s_{l}^{\prime}\right \}
\]
and $\beta_{n\left(  k\right)  }^{\prime}|_{\left(  t_{-k},t_{k}\right)  }$
does not contain a singularity $s_{i},i\leq l.$ Clearly, $\gamma_{n}^{\prime
}\equiv \beta_{n\left(  k\right)  }^{\prime}|_{\left[  t_{-k},t_{k}\right]  }$
is a sequence of generalized saddles (each with endpoints $\beta_{n\left(
k\right)  }^{\prime}\left(  t_{-k}\right)  $, $\beta_{n\left(  k\right)
}^{\prime}\left(  t_{k}\right)  )$ approximating $\beta^{\prime}.$ It follows
that $\gamma_{n}=\psi \left(  \gamma_{n}^{\prime}\right)  $ is a sequence of
generalized saddles approximating $\beta.$
\end{proof}

\end{document}